\documentclass[english]{amsart}

\usepackage{amssymb,amsmath,amsthm,amsrefs,esint,braket,mathtools,enumitem,hyperref,comment,xcolor,doi,microtype}

\usepackage[margin=25mm]{geometry}

\usepackage[capitalise]{cleveref}

\usepackage[T1]{fontenc}

\providecommand{\abs}[1]{\left\vert #1 \right\vert}
\providecommand{\norm}[1]{\left\Vert #1 \right\Vert}

\providecommand{\pt}[1]{\left( #1 \right)}
\providecommand{\spt}[1]{\left[ #1 \right]}

\newcommand{\dist}{\mathrm{dist}}
\newcommand{\RR}{\mathbb R}

\newcommand{\NN}{\mathbb N}
\newcommand{\Ss}{\mathbb S}

\newcommand{\D}{\,\mathrm d}

\newcommand{\dx}{{\D x}}

\newcommand{\ve}{\varepsilon}

\newtheorem{theorem}{Theorem}[section]
\newtheorem{proposition}{Proposition}[section]

\newtheorem{corollary}{Corollary}[section]

\newtheorem*{conjecture}{Conjecture}

\Crefname{corollary}{Corollary}{Corollaries}
\Crefname{lemma}{Lemma}{Lemmas}
\Crefname{theorem}{Theorem}{Theorems}
\Crefname{proposition}{Proposition}{Propositions}
\Crefname{definition}{Definition}{Definitions}

\begin{document}

\title{A critical Hardy-Rellich inequality}

\author{Hern{\'a}n Castro}
\address{Instituto de Matem{\'a}ticas, Universidad de Talca, Casilla 747, Talca, Chile}
\email{hcastro@utalca.cl}
\date{\today}

  \begin{abstract}
    In this work, we prove a critical version of a Hardy-Rellich type inequality. We show that for \(N\geq 1\) and for \(-N<a<1\), \(a\neq 0\) there exists a constant \(C_{N,a}>0\) such that
    \[
      \int_{\mathbb R^N}\left|\nabla\left(\frac{u(x)}{|x|}\right)\right|^N|x|^a\,\mathrm{d}x\leq C_{N,a}\int_{\mathbb R^N}\left|\Delta u(x)\right|^N |x|^a\,\mathrm{d}x,
    \]
    for any \(u\in C^\infty_c(\mathbb R^N\setminus\left\{0\right\})\).
  \end{abstract}


\maketitle

\section{Introduction}

Suppose \(N\geq 1\) and that \(u\) belongs to \(C^\infty_c(\RR^N)\). The classical Hardy inequality \cite{HL52} says that for any \(1\leq p< N\) one has
\begin{equation}\label{hardy1}
  \int_{\RR^N}\frac{\abs{u(x)}^p}{{\abs x}^p}\dx\leq \pt{\frac{p}{N-p}}^p\int_{\RR^N} \abs{\nabla u(x)}^p\dx,
\end{equation}
and if in addition one supposes that \(u\in C^\infty_c(\RR^N\setminus \set{0})\) then \eqref{hardy1} can be improved in the sense that it holds for all \(p\neq N\), that is
\begin{equation}\label{hardyp}
  \int_{\RR^N}\frac{\abs{u(x)}^p}{{\abs x}^p}\dx\leq \abs{\frac{p}{N-p}}^p\int_{\RR^N} \abs{\nabla u(x)}^p\dx.
\end{equation}

If one allows higher order derivatives on the right hand side of \eqref{hardyp}, we have the so called Rellich inequalities \cite{Rell1956,Rell1969} initially proved by Rellich in \cite{Rell1956} and later generazlized by Davies and Hinz in \cite{DaHi98}: if \(N\geq 3\), \(1<p<N/2\), and \(u\in C^\infty_c(\RR^N\setminus \set{0})\) one has
\begin{equation}\label{rellich1}
  \int_{\RR^N}\abs{\frac{u(x)}{\abs x^{2}}}^p\dx\leq \pt{\frac{p^2}{N(N-2p)(p-1)}}^p\int_{\RR^N} \abs{\Delta u(x)}^p\dx.
\end{equation}
It is not our intention to give a complete survey of the results related to Hardy and Rellich inequalities as the past and recent literature is very rich on the subject; however we would like to point the interested reader to the collective works in \cite{KO90,Oka96,BM97,Mit00,GGM2004}, and the references therein. In those references one can find interesting remarks on the subject of best constants and improvements on the inequalities \eqref{hardyp} and \eqref{rellich1}.

Another family of inequalities are the so called Hardy-Rellich inequalities. These inequalities appear when one tries to estimate terms involving \(\nabla u\) on the left hand side of \eqref{rellich1}. For example, take \(u\in C^\infty_c(\RR^N\setminus\set{0})\), and apply \eqref{hardyp} to a suitable regularization of \(v=\abs{\nabla u}\), then one can obtain for \(p\neq N\) the inequality
\begin{equation}\label{rellichp-}
  \int_{\RR^N}\frac{\abs{\nabla u(x)}^p}{\abs x^p}\dx\leq C_{N,p}\int_{\RR^N} \abs{D^2u(x)}^p\dx,
\end{equation}
if, in addition, \(p>1\), we can use elliptic \(L^p\) theory in \eqref{rellichp-} to deduce that
\begin{equation}\label{rellichp}
  \int_{\RR^N}\frac{\abs{\nabla u(x)}^p}{\abs x^p}\dx\leq C_{N,p}\int_{\RR^N} \abs{\Delta u(x)}^p\dx.
\end{equation}
The constant \(C_{N,p}\) obtained by this method is not optimal, in fact, the optimal constant for \(p=2\) and \(N\geq 3\) is known (see \cite[Theorem 4.2]{TZ07}; see also \cite[Theorem 1.2]{Mus2014} and \cite{Caz2020}), that is
\begin{equation}\label{rellich2-5}
  \int_{\RR^N}\frac{\abs{\nabla u(x)}^2}{\abs x^2}\dx\leq C_N\int_{\RR^N} \abs{\Delta u(x)}^2\dx,
\end{equation}
where
\[
  C_N=
  \begin{dcases}
    \frac{4}{N^2}&\text{if }N\geq 5,\\
    \frac13&\text{if }N=4,\\
    \frac{36}{25}&\text{if }N=3.
  \end{dcases}
\]

Also, a generalization of \eqref{rellichp} was studied in \cite{Cos2009,DJSJ13,HaTa2025}. In \cite{Cos2009}, the author connected \eqref{rellichp} with the Caffarelli-Kohn-Nirenberg \cite{CKN1984} family of inequalities and he showed that for \(p=2\) one has
\[
  \int_{\RR^N}\frac{\abs{\nabla u(x)}^2}{\abs x^{a+b+1}}\dx
  \leq C_{N,a,b}
  \pt{\int_{\RR^N} \frac{\abs{\Delta u(x)}^2}{\abs x^{2a}}\dx}^\frac12
  \pt{\int_{\RR^N}\frac{\abs{\nabla u(x)}^2}{\abs x^{2b}}\dx}^\frac12,
\]
for suitable \(a,b\in\RR\). Later in \cite{DJSJ13} the above inequality was generalized for \(1<p<N\) and \((p-N)/(p-1)\leq a+b+1 \leq 0\) to obtain
\[
  \int_{\RR^N}\frac{\abs{\nabla u(x)}^p}{\abs{x}^{a+b+1}}\dx
  \leq C_{N,p,a,b}
  \pt{\int_{\RR^N}\frac{\abs{\Delta_p u}^p}{\abs{x}^{ap}}\dx}^{\frac1p}
  \pt{\int_{\RR^N}\frac{\abs{\nabla u(x)}^p}{\abs{x}^{\frac{bp}{p-1}}}\dx}^{1-\frac1p},
\]
where \(\Delta_p u\) denotes the \(p\)-Laplace operator. On the other hand, in \cite{HaTa2025} a weighted versión of \eqref{rellichp} is analyzed when \(p=2\) for vector fields \(\mathbf u\).

Observe that in the above inequalities, the case \(p=N\) is missing. This is no coincidence, as the classical Hardy inequalities \eqref{hardy1} and \eqref{hardyp} are known to fail for \(p=N\) unless one includes some correction term in the left-hand side (for instance, see \cite[Lemma 2.3]{ACR02}, and \cite{AS2002,Tak2015} where a logarithmic term is added to the left-hand side in the case of bounded domains). Additionally, one can easily see that from the fact that \eqref{hardy1} fails for \(p=N\), one can deduce that no inequality of the form \eqref{rellichp} can be proven for \(p=N\).

Similarly, it has been shown that a version of \eqref{rellich1} for the case of \(p=N=2\) cannot be proven; in fact one has (see \cite[Theorem 1 in \S 7]{Rell1969})
\begin{equation}\label{rellich2}
  \inf_{\substack{u\in C^2_c(\RR^2\setminus\set{0})\\ u\not\equiv 0}}\frac{\displaystyle\int_{\RR^2} \abs{\Delta u(x)}^2\dx}{\displaystyle\int_{\RR^2}\frac{\abs{u(x)}^2}{\abs x^{4}}\dx}=0.
\end{equation}
However, and observation from \cite{CaMu2012} tells us that, among other interesting things, the infimum in \eqref{rellich2} can be made positive and explicit if one restricts the set of functions to the set of \emph{radially symmetric functions}.

The above remarks tell us that in order to obtain a Hardy-Rellich type inequality for the case \(p=N\) one could restrict the set of functions to a particular class, or perhaps one could modify the left-hand side of the respective inequality. A related approach comes from studying the difference between the Hardy term \(\frac{\nabla u}{\abs x}\) and the Rellich term \(\frac{u}{\abs x^2}\). In \cite{CW10} we studied that phenomenon and we considered the difference
\[
  \frac{u'(x)}{x}-\frac{u(x)}{x^2}=\pt{\frac{u(x)}{x}}'
\]
and we were able to show that
\begin{equation}\label{1D-HardyRellich}
  \int_0^\infty\abs{\pt{\frac{u(x)}{x}}'}\dx\leq \int_0^\infty \abs{u''(x)}\dx
\end{equation}
despite the possible non-integrability of \(\frac{u'}{x}\) and \(\frac{u}{x^2}\) separately. This cancellation phenomenon was later generalized in \cite{CDW11,CDW11-2} to the following surprising Hardy inequality for \(p=1\): If \(d:\Omega\to\RR\) is a suitable regularization of \(\delta(x)=\dist(x,\partial\Omega)\) for a smooth domain \(\Omega\), then
\begin{equation}\label{hardy-bdry}
  \int_\Omega\abs{\nabla\pt{\frac{u(x)}{d(x)}}}\dx\leq C\norm{u}_{W^{2,1}(\Omega)},
\end{equation}
and a similar result for any \(u\in W^{m,1}_0(\Omega)\). Later, in \cite{CiNo2017} the authors generalized \eqref{hardy-bdry} and they obtained an inequality of the form
\[
  \norm{u d^{k-m}}_{W^{k,p}(\Omega,d^{p-1}\dx)}\leq C\norm{\nabla u}_{W^{m,p}(\Omega, d^{p-1}\dx)}
\]
for \(1\leq k\leq m-1\) and any \(u\in W^{m,p}_0(\Omega)\).

The above cancellation phenomena discovered in \cite{CW10,CDW11,CDW11-2} lead us to believe that a similar cancellation might occur when dealing with
\[
  \nabla\pt{\frac{u(x)}{\abs x}}=\frac{1}{\abs{x}}\nabla u(x)-\frac{u(x)}{\abs{x}^3}x,
\]
and this is the content of the following

\begin{conjecture}
  For every \(N\geq 1\) there exists a constant \(C_N>0\) such that
  \begin{equation}\label{hardy-rellichN}
    \int_{\RR^N}\abs{\nabla\pt{\frac{u(x)}{\abs x}}}^N\dx\leq C_N\int_{\RR^N}\abs{\Delta u(x)}^N\dx
  \end{equation}
  for all \(u\in C^\infty_c(\RR^N\setminus\set{0})\).
\end{conjecture}

However this conjecture is false as one can consider a suitable approximation of the funcion \(u(x)=x_1\) by functions in \(u_n\in C^\infty_c(\RR^N\setminus \set{0})\). In the limit it is not difficult to see that
\[
  \int_{\RR^N}\abs{\Delta u_n(x)}^N\dx \xrightarrow[n\to\infty]{} 0
\]
while
\[
  \int_{\RR^N}\abs{\nabla\pt{\frac{u_n(x)}{\abs x}}}^N\dx \xrightarrow[n\to\infty]{} +\infty.
\]

However, one can still prove something along the lines of \eqref{1D-HardyRellich} for higher dimensions if one adds suitable weights to each integral, as the following theorem shows.

\begin{theorem}\label{main-thm}
  Let \(N\geq 2\) and \(a\in (-\infty,0)\cup (0,1)\), then there exists a constant \(C_{N,a}>0\) such that
  \begin{equation}\label{hardy-rellichN2}
    \int_{\RR^N}\abs{\nabla\pt{\frac{u(x)}{\abs x}}}^N\abs{x}^{a}\dx
    \leq C_{N,a}
    \int_{\RR^N}\abs{D^2 u(x)}^N\abs{x}^a\dx,
  \end{equation}
  for all \(u\in C^\infty_c(\RR^N\setminus\set{0})\).
\end{theorem}
and the following corollary

\begin{corollary}\label{main-corollary}
  If \(-N<a<1\) and \(a\neq 0\), then there exists a constant \(C_{N,a}>0\) such that
  \begin{equation}\label{hardy-rellichN3}
    \int_{\RR^N}\abs{\nabla\pt{\frac{u(x)}{\abs x}}}^N\abs{x}^{a}\dx
    \leq C_{N,a}
    \int_{\RR^N}\abs{\Delta u(x)}^N\abs{x}^a\dx,
  \end{equation}
  for all \(u\in C^\infty_c(\RR^N\setminus\set{0})\).
\end{corollary}

The constants in the above theorems are not given explicitly. This is no coincidence as the method used in the proof of the results do not give optimal constants. Therefore we have a natural open problem

\noindent\textbf{Open problem 1}: To find the best possible constant in both \eqref{hardy-rellichN} and \eqref{hardy-rellichN3}.

Additionally, in light of \cref{main-corollary} one could raise a couple of questions. For instance by looking at the work of Herbst \cite{Her1977} for the fractional Laplacian \((-\Delta)^s\) we know that
\[
  \int_{\RR^N}\abs{\frac{u(x)}{\abs x^{2s}}}^p\dx
  \leq C_{N,p,s}
  \int_{\RR^N} \abs{(-\Delta)^s u(x)}^p\dx
\]
holds for \(N>2ps\). Also, in the recent work \cite{DeNDj2024} the following generalization was obtained for \(N>(\theta+2s)p>0\)
\[
  \int_{\RR^N}\abs{\frac{u(x)}{\abs x^{\theta+2s}}}^p\dx
  \leq C_{N,p,s,\theta}
  \int_{\RR^N} \abs{\frac{(-\Delta)^s u(x)}{\abs x^{\theta}}}^p\dx.
\]

A natural question is then

\noindent\textbf{Open problem 2}: What can be said for \(N=(\theta+2s)p\) in the above inequalities??

Or in a more general setting

\noindent\textbf{Open problem 3}: Does a fractional Hardy-Rellich inequality of the form
\begin{equation}\label{open-p}
  \int_{\RR^N}\abs{\nabla\pt{\frac{u(x)}{\abs x^{\alpha}}}}^p\abs{x}^\beta\dx
  \leq C
  \int_{\RR^N} \abs{(-\Delta)^s u(x)}^p\abs x^{\gamma}\dx
\end{equation}
hold for suitable conditions on the parameters \(s,\alpha,\beta,\gamma\)? and if such an inequality is true, is there a critical exponent \(p\)?

The rest of this article is as follows: in \cref{sect-N=1} we generalize the idea from \cite{CW10} in such a way that it can be used in \cref{sect-proof-thm} to prove our main result.

\section{A result in dimension \texorpdfstring{\(N=1\)}{N=1}}\label{sect-N=1}

Recall the main result from \cite{CW10}: if \(u\in C^\infty_c(0,\infty)\) then
\begin{equation}\label{baby-esti}
  \int_0^\infty\abs{\pt{\frac{u(x)}{x}}'}\dx\leq \int_0^\infty \abs{u''(x)}\dx.
\end{equation}
The proof of \eqref{baby-esti} is quite simple once one realizes that any \(u\in C^\infty_c(0,\infty)\) satisfies the identity
\[
  \pt{\frac{u(x)}{x}}'=\frac{1}{x^2}\pt{xu'(x)-u(x)}=\frac{1}{x^2}\int_0^xsu''(s)\D s,
\]
and thus,
\[
  \int_0^\infty\abs{\pt{\frac{u(x)}{x}}'}\dx\leq \int_0^\infty \frac{1}{x^2}\int_0^xs\abs{u''(s)}\D s\dx=\int_0^\infty \abs{u''(s)}\D s
\]
thanks to Tonelli's theorem.

Observe that the above proof tells us that the operator
\[
  Tf(x)=\frac{1}{x^2}\int_0^x sf(s)\D s,
\]
is bounded in \(L^1(0,\infty)\) with \(\norm{Tf}_{L^1}\leq \norm{f}_{L^1}\), but in fact one has more
\begin{proposition}\label{key-prop}
  For \(a,b\in\RR\), define the operator \(T_{a,b}\) as
  \begin{equation}\label{operator-T}
    T_{a,b}f(x)=\frac{1}{x^a}\int_0^x s^bf(s)\D s.
  \end{equation}
  If \(p,\alpha\in\RR\) satisfy \(p\geq 1\) and \(p(a-1)>\alpha\), then \(T_{a,b}\) is bounded from \(L^p(x^\alpha\dx)\) to \(L^p(x^{\alpha-p(a-b-1)}\dx)\) with
  \[
    \int_0^\infty \abs{T_{a,b}f(x)}^px^\alpha\dx
    \leq \frac{1}{p(a-1)-\alpha}
    \int_0^\infty\abs{f}^px^{\alpha-p(a-b-1)}\dx.
  \]
\end{proposition}

\begin{proof}
  We use Jensen's inequality and Tonelli's theorem to write
  \begin{align*}
    \int_0^\infty \abs{T_{a,b}f(x)}^px^\alpha\dx
    &=
    \int_0^\infty \abs{\frac{1}{x^a}\int_0^x s^bf(s)\D s}^px^\alpha\dx\\
    &=
    \int_0^\infty \abs{\frac{1}{x}\int_0^x s^bf(s)\D s}^px^{\alpha-p(a-1)}\dx\\
    &\leq
    \int_0^\infty \int_0^x s^{bp}\abs{f(s)}^px^{\alpha-p(a-1)-1}\D s\dx\\
    &=
    \int_0^\infty \int_s^\infty s^{bp}\abs{f(s)}^qx^{\alpha-p(a-1)-1}\dx\D s\\
    &=\frac1{p(a-1)-\alpha}
    \int_0^\infty s^{\alpha-p(a-b-1)}\abs{f(s)}^p\D s\\
  \end{align*}
\end{proof}

From \cref{key-prop} we easily obtain this generalization of the result from \cite{CW10}
\begin{corollary}
  Suppose \(n,k\in\NN\cup\set{0}\), \(p\geq 1\) and that \(u\in C^\infty_c(0,\infty)\). If \(pn>\alpha\) then
  \[
    \int_0^\infty \abs{\frac{d^{n}}{dx^{n}}\pt{\frac{\frac{d^{k}u}{dx^{k}}(x)}{x}}}^px^\alpha\dx
    \leq
    \frac{1}{pn-\alpha}\int_0^\infty \abs{\frac{d^{n+k+1}u}{dx^{n+k+1}}(x)}^px^\alpha\dx
  \]
\end{corollary}
\begin{proof}
  One has \(\pt{\frac{u(x)}{x}}'=T_{2,1}(u'')\) from where a direct induction argument tells us that
  \[
    \frac{d^{n}}{dx^{n}}\pt{\frac{u(x)}{x}}=T_{n+1,n}\pt{\frac{d^{n+1}u}{dx^{n+1}}}.
  \]
  If we apply the above to \(\frac{d^{k}u}{dx^{k}}\) instead of \(u\) we get
  \[
    \frac{d^{n}}{dx^{n}}\pt{\frac{\frac{d^{k}u}{dx^{k}}(x)}{x}}=T_{n+1,n}\pt{\frac{d^{n+k+1}u}{dx^{n+k+1}}}
  \]
  and the result follows from the boundedness of \(T_{n+1,n}\).
\end{proof}

\section{The main result}\label{sect-proof-thm}

The key ingredient in the proof of the main theorem is the following observation: in spherical coordinates \(x=r\omega\), \(r>0\) and \(\omega\in \Ss^{N-1}\), one can write
\begin{equation*}
  \nabla u(r\omega)=\partial_r u(r\omega)\omega+\frac1r\nabla_{\Ss^{N-1}}u(r\omega),
\end{equation*}
which after differentiating in the \(r\) variable gives
\begin{equation*}
  \partial_r\nabla u(r\omega)=D^2u(r\omega)\omega=\partial_{rr}u(r\omega)\omega+\partial_r\pt{\frac1r\nabla_{\Ss^{N-1}}u(r\omega)}.
\end{equation*}
Since the vectors \(\omega\) and \(\nabla_{\Ss^{N-1}}u(r\omega)\) are orthogonal, we have
\begin{equation}\label{key-ineq}
  \abs{D^2u(r\omega)}^2\geq \abs{\partial_{rr}u(r\omega)}^2+\abs{\partial_r\pt{\frac1r\nabla_{\Ss^{N-1}}u(r\omega)}}^2.
\end{equation}

\begin{proof}[Proof of \cref{main-thm}]
  we have
  \begin{align*}
    \abs{\nabla\pt{\frac{u(x)}{\abs x}}}^2&=\abs{\frac{1}{\abs{x}}\nabla u(x)-\frac{u(x)}{\abs{x}^3}x}^2\\
    &=\frac{\abs{\nabla u(x)}^2}{\abs{x}^2}+\frac{\abs{u(x)}^2}{\abs{x}^4}-2\frac{u(x)}{\abs{x}^4}x\cdot\nabla u(x)\\
    &=\pt{\frac{\abs{\nabla u(x)}^2}{\abs{x}^2}-\frac{\abs{x\cdot\nabla u(x)}^2}{\abs{x}^4}}+\frac{1}{\abs{x}^4}\abs{x\cdot\nabla u(x)-u(x)}^2.
  \end{align*}

  Writing the above in spherical coordinates \(x=r\omega\) gives
  \[
    \abs{r\omega\cdot \nabla u(r\omega)}^2=\abs{r\omega}^2\abs{\partial_r u(r\omega)}^2+\abs{\nabla_{\Ss^{N-1}}u(r\omega)}^2,
  \]
  so that
  \begin{align*}
    \abs{\nabla\pt{\frac{u(x)}{\abs x}}}^2
    &=\pt{\frac{\abs{\nabla u(x)}^2}{\abs{x}^2}-\frac{\abs{x\cdot\nabla u(x)}^2}{\abs{x}^4}}
    +\frac{1}{\abs{x}^4}\abs{x\cdot\nabla u(x)-u(x)}^2\\
    &=\frac1{r^4}\abs{\nabla_{\Ss^{N-1}} u(r\omega)}^2
    +\frac{1}{r^4}\abs{r\partial_r u(r\omega)-u(r\omega)}^2.
  \end{align*}
  Using the convexity of the function \(t\mapsto t^{\frac{N}2}\) for \(N\geq 2\) we deduce that
  \begin{equation}\label{first-step}
    \abs{\nabla\pt{\frac{u(x)}{\abs x}}}^N
    \leq C\spt{
      \frac1{r^{2N}}\abs{\nabla_{\Ss^{N-1}} u(r\omega)}^N
      +
      \pt{\frac{1}{r^{2}}\abs{r\partial_r u(r\omega)-u(r\omega)}}^N
    }
  \end{equation}
  for some constant \(C>0\) depending solely on the dimension \(N\).

  Hence, if \(a<1\) we can estimate the ``radial'' term using \cref{key-prop}, this time for \(p=N\), \(a=2,~b=1\) and \(\alpha=N+a-1\)
  \begin{align*}
    \int_{\Ss^{N-1}}\int_0^\infty
    &\pt{\frac{1}{r^2}\abs{r\partial_ru(r\omega)-u(r\omega)}}^Nr^{N+a-1}\D r\D\omega\\
    &=
    \int_{\Ss^{N-1}}\int_0^\infty
    r^{N+a-1}\abs{T_{2,1}(\partial_{rr}u)(r\omega)}^N\D r\D\omega\\
    &\leq
    \frac{1}{1-a}
    \int_{\Ss^{N-1}}\int_0^\infty
    \abs{\partial_{rr}u(s\omega)}^N s^{N+a-1}\D s\D\omega\\
    &\leq C
    \int_{\RR^N}
    \abs{D^2 u(x)}^N\abs{x}^a\dx,
  \end{align*}
  for some constant \(C>0\) depending on the dimension \(N\) and \(a<1\). Here we have use the inequality \eqref{key-ineq} to estimate the last integral.

  For the ``angular'' term we use the following version of the one dimensional Hardy inequality: if \(f\in C^\infty_c(0,\infty)\), \(a\neq 0\) and \(p>1\) we have
  \begin{equation}
    \int_0^\infty r^{a-1}\abs{f(r)}^p\D r\leq \abs{\frac{p}{a}}^{p}\int_0^\infty r^{p+a-1}\abs{f'(r)}^p\D r,
  \end{equation}
  which we apply to the function
  \(f_\ve(r)=(\frac1{r^2}\abs{\nabla_{\Ss^{N-1}}u(r\omega)}^2+\ve^2)^{\frac12}\) where \(\omega\in \Ss^{N-1}\) is fixed. Observe that one has
  \[
    \abs{f_\ve'(r)}\leq \abs{\partial_r\pt{\frac1r\nabla_{\Ss^{N-1}}u(r\omega)}}
  \]
  hence we get
  \[
    \int_0^\infty r^{a-1}\abs{f_\ve(r)}^N\D r\leq C_{a,N}\int_0^\infty r^{N+a-1}\abs{f_\ve'(r)}^N\D r
  \]
  and from Fatou's lemma we conclude that
  \[
    \int_0^\infty r^{a-N-1}\abs{\nabla_{\Ss^{N-1}}u(r\omega)}^N\D r\leq C_{a,N}\int_0^\infty r^{N+a-1}\abs{\partial_r\pt{\frac1r\nabla_{\Ss^{N-1}}u(r\omega)}}^N\D r
  \]
  for each \(\omega\in \Ss^{N-1}\). Now we write
  \begin{align*}
    \int_{\Ss^{N-1}}\int_0^\infty
    \frac1{r^{2N}}&\abs{\nabla_{\Ss^{N-1}} u(r\omega)}^Nr^{N+a-1}\D r\D\omega\\
    &=
	\int_{\Ss^{N-1}}\int_0^\infty
    r^{a-N-1}\abs{\nabla_{\Ss^{N-1}} u(r\omega)}^N\D r\D\omega\\
    &\leq C_{a,N}
    \int_{\Ss^{N-1}}\int_0^\infty
    \abs{\partial_r\pt{\frac1r\nabla_{\Ss^{N-1}} u(r\omega)}}^Nr^{N+a-1}\D r\D\omega\\
    &\leq C_{a,N}
    \int_{\Ss^{N-1}}\int_0^\infty
    \abs{D^2 u(r\omega)}^Nr^{N+a-1}\D r\D\omega\\
    &= C_{a,N}
    \int_{\RR^N}
    \abs{D^2 u(x)}^N\abs{x}^a\dx,
  \end{align*}
  thanks to \eqref{key-ineq} and the change of variables \(x=r\omega\).
\end{proof}

Finally, to prove \cref{main-corollary} we recall a result about the boundedness of singular integrals in a weighted setting. In the classical work of Coifman and Fefferman \cite{CoiFe1974}, it is shown that if \(w\) a Muckenhoupt weight, then singular integrals operators are bounded in \(L^p(w\dx)\), more precisely

\begin{theorem}[Theorem III in \cite{CoiFe1974}]
  Suppose that \(w\in A_\infty=\bigcup_{q>1}A_q\), that is, there exist constants \(C,\delta>0\) such that for every cube \(Q\) and measurable \(E\subset Q\) one has
  \[
    \frac{w(E)}{w(Q)}\leq C\pt{\frac{\abs{E}}{\abs{Q}}}^\delta,
  \]
  then
  \[
    \int_{\RR^N}\abs{Tf(x)}^pw(x)\dx\leq C_{N,p,w}\int_{\RR^N}\abs{f(x)}^pw(x)\dx,
  \]
  where \(Tf=K\ast f\), and \(K\) is a standard convolution kernel.
\end{theorem}

In particular, for each \(i\in\set{1,\ldots, N}\) the kernel \(K_i(x)=C_N   \frac{x_i}{\abs{x}^{N+1}}\) satisfies the standard conditions of size, smoothness and cancellation and the operator \(T_i u=K_i\ast u\) is the Riesz transform \(R_i\), in particular one has
\[
  \partial_{ij}u=-R_iR_j\Delta u,
\]
and as a consequence we readily obtain
\begin{corollary}
  If \(w\in A_\infty\), there exists a constant \(C=C(N,p,w)>0\) such that
  \[
    \int_{\RR^N}\abs{\partial_{ij}u(x)}^pw(x)\dx
    \leq C
    \int_{\RR^N}\abs{\Delta u}^pw(x)\dx,
  \]
  for every \(u\in C^\infty_c(\RR^N)\).
\end{corollary}

In particular, when \(w(x)=\abs{x}^a\) we get

\begin{corollary}\label{CZ-estimate}
  If \(a>-N\), then there exists a constant \(C=C(a,p,N)>0\) such that
  \[
    \int_{\RR^N}\abs{D^2 u(x)}^N\abs{x}^a\dx
    \leq C_{N,p,w}
    \int_{\RR^N}\abs{\Delta u}^N\abs{x}^a\dx
  \]
  for every \(u\in C^\infty_c(\RR^N)\).
\end{corollary}
\begin{proof}
  This follows because the weight \(w(x)=\abs{x}^a\in A_q\) if and only if \(-N<a<N(q-1)\) (\cite[p. 298]{HeKiMa2006}), and because \(A_\infty=\bigcup_{q>1}A_q\) (\cite[Lemma 15.8]{HeKiMa2006}) we conclude that \(w\in A_\infty\) if and only if \(a>-N\).
\end{proof}

In conclusion, combining \cref{CZ-estimate} and \cref{main-thm} readily yields \cref{main-corollary}.

\section*{Acknowledgements}

I would like to thank Shuto Hasegawa for pointing out a mistake in a previous version of this manuscript.

\bibliographystyle{amsplain}


\end{document}